\documentclass[11pt]{article}
\usepackage{amsmath,amsfonts,latexsym, xspace,amsthm,graphicx, amssymb}
\usepackage{enumerate}
\usepackage[letterpaper,left=1.5in,right=1.5in,top=1.5in,bottom=1.5in]{geometry}

\usepackage{url}

\usepackage{mathtools}
\mathtoolsset{showonlyrefs}

\usepackage{color}

\newcommand{\ind}[1]{{\bf 1}_{#1}}
\newcommand{\E}[1]{\mathbb{E} \left(#1\right)}
\newcommand{\Ex}[2]{\mathbb{E}_{#1} \left(#2\right)}
\newcommand{\Var}[1]{\text{Var}(#1)}

\newcommand{\ang}[1]{\langle #1 \rangle}

\def\AA{\mathcal{A}}
\def\BB{\mathcal{B}}

\def\EE{\mathcal{E}}

\def\LL{\mathcal{L}}

\def\PP{\mathcal{P}}

\def\XX{\mathcal{X}}

\def\ol{\overline}

\def\AND{\text{ and }}

\def\F{\Phi}

\def\d{\delta}

\def\e{\varepsilon}
\def\f{\phi}
\def\g{\gamma}
\def\G{\Gamma}

\def\th{\theta}

\def\l{\lambda}
\def\m{\mu}

\def\r{\rho}

\def\t{\tau}
\def\om{\omega}
\def\OM{\Omega}
\newcommand\Prob[1]{{\mbox{Pr}\left\{#1\right\}}}

\newcommand\beq[2]{\begin{equation}\label{#1}#2\end{equation}}

\newtheorem{lemma}{Lemma}
\newtheorem{theorem}{Theorem}
\newtheorem{corollary}{Corollary}
\newtheorem{definition}{Definition}
\newtheorem{claim}{Claim}

\newcommand{\bfrac}[2]{\left({\frac{#1}{#2}}\right)}

\title{The cover time of a biased random walk on a random regular graph of odd degree}

\author{Tony Johansson\thanks{Department of Mathematics, Uppsala University, Uppsala, Sweden. Supported in part by the Knut and Alice Wallenberg Foundation. Email: tony.johansson@math.uu.se}}

\begin{document}
\maketitle

\begin{abstract}
We consider a random walk process which prefers to visit previously unvisited edges, on the random $r$-regular graph $G_r$ for any odd $r\geq 3$. We show that this random walk process has asymptotic vertex and edge cover times $\frac{1}{r-2}n\log n$ and $\frac{r}{2(r-2)}n\log n$, respectively, generalizing the result from \cite{CFJ} from $r = 3$ to any larger odd $r$. This completes the study of the vertex cover time for fixed $r\geq 3$, with \cite{bcf} having previously shown that $G_r$ has vertex cover time asymptotic to $\frac{rn}{2}$ when $r\geq 4$ is even.
\end{abstract}

\section{Introduction}

We consider a biased random walk on the random $r$-regular $n$-vertex graph $G_r$ for any odd fixed $r\geq 5$, i.e. a graph chosen uniformly at random from the set of $r$-regular graph on an even number $n$ of vertices. In short, this is a random walk which chooses a previously unvisited edge whenever possible, and otherwise chooses an edge uniformly at random. See Section \ref{sec:outline} for a precise definition. In \cite{CFJ} it is shown that with high probability, $G_3$ is such that the expected vertex cover time $C_V^b(G_3)$ and expected edge cover time $C_E^b(G_3)$ of the biased random walk satisfy\footnote{We say that $a_n \sim b_n$ if $\lim a_n / b_n = 1$.}
$$
C_V^b(G_3) \sim n\log n, \quad C_E^b(G_3) \sim \frac32 n\log n.
$$
We generalize this result as follows.
\begin{theorem}\label{thm:whptheorem}
Suppose $r\geq 3$ is odd, and let $G_r$ be chosen uniformly at random from the set of $r$-regular graphs on $n$ vertices. Then with high probability, $G_r$ is such that
$$
C_V^b(G_r) \sim \frac{1}{r-2}n\log n, \quad C_E^b(G_r) \sim \frac{r}{2(r-2)}n\log n.
$$
\end{theorem}
With this the asymptotic leading term of $C_V^b(G_r)$ is known for all $r \geq 3$, with Berenbrink, Cooper and Friedetzky \cite{bcf} having previously shown that $C_V^b(G_r) \sim \frac{rn}{2}$ for any even $r\geq 4$. They also showed that for even $r$, $C_E^b(G_r) = O(\om n)$ for any $\om$ tending to infinity with $n$, with the $\om$ factor owing to the w.h.p.\footnote{An event $\EE$ holds {\em with high probability} (w.h.p.) if $\Prob{\EE} \to 0$ as $n\to\infty$.} existence of cycles of length up to $\om$.

Cooper and Frieze \cite{CFreg} considered the simple random walk on $G_r$, showing that for any $r\geq 3$, $C_V^s(G_r) \sim \frac{r-1}{r-2}n\log n$ and $C_E^s(G_r) \sim \frac{r(r-1)}{2(r-2)}n\log n$, and we see that the biased random walk speeds up the cover time by a factor of $1/(r-1)$ for odd $r$. Cooper and Frieze \cite{CFNet} also consider the non-backtracking random walk, i.e. the walk which at no point reuses the edge used in the previous step, showing that $C_V^{nb}(G_r) \sim n\log n$ and $C_E^{nb}(G_r)\sim \frac{r}{2}n\log n$. Here, the biased random walk gains a factor of $1/(r-2)$ for odd $r$. 

Theorem \ref{thm:whptheorem} will follow from the following theorem. Let $C_V^b(G; s)$ ($C_E^b(G; t)$) denote the expected time taken for the biased random walk to visit $s$ vertices ($t$ edges).  Note that $C_\cdot^b(G; \cdot)$ is defined as an expectation over the space of random walks on the fixed graph $G$, and that $\E{C_\cdot^b(G_r; \cdot)}$ takes the expectation of $C_\cdot^b(G; \cdot)$ when $G$ is chosen uniformly at random from the set of $r$-regular graphs. 
\begin{theorem}\label{thm:exptheorem}
Suppose $r\geq 3$ is odd, and suppose $G_r$ is chosen uniformly at random from the set of $r$-regular graphs on an even number $n$ of vertices. Let $n - n\log^{-2} n \leq s \leq n$ and $(1-\log^{-2}n)\frac{rn}{2} \leq t \leq rn/2$, and let $\e > 0$. Then
\begin{align}
\E{C_V^b(G_r; s)} & = \frac{1\pm\e}{r-2} n\log\bfrac{n}{n-s+1} + o(n\log n), \\
\E{C_E^b(G_r; t)} & = \frac{r\pm\e}{2(r-2)} n\log\bfrac{rn}{rn-2t+1} + o(n\log n).
\end{align}
\end{theorem}
We take $a = b\pm c$ to mean that $b-c < a < b + c$. The $(1-\log^{-2}n)$ factor in the lower bounds for $s, t$ is a fairly arbitrary choice, and the proof here is valid for any $(1-1/\om)$ factor with $\om$ tending to infinity sufficiently slowly. The specific choice of $\log^{-2} n$ is made to aid readability.

Applying Theorem \ref{thm:exptheorem} with $s = n$ and $t = rn/2$ gives $\E{C_V^b(G_r)} \sim \frac{1}{r-2}n\log n$ and $\E{C_E^b(G_r)} \sim \frac{r}{2(r-2)}n\log n$. A little extra work is needed to conclude that w.h.p. $G_r$ is such that $C_V^b(G_r), C_E^b(G_r)$ have the same asymptotic values. We refer to the full paper version of \cite{CFJ}, where this is done in detail.

\section{Proof outline}\label{sec:outline}

The random $r$-regular graph $G_r$ is chosen according to the {\em configuration model}, introduced by Bollob\'as \cite{bollobas}. Each vertex $v\in [n]$ is associated with a set $\mathcal{P}(v)$ of $r$ {\em configuration points}, and we let $\mathcal{P} = \cup_{v} \mathcal{P}(v)$. We choose u.a.r. ({\em uniformly at random}) a perfect matching $\mu$ of the points in $\mathcal{P}$. Each $\mu$ induces a multigraph $G$ on $[n]$ in which $u$ is adjacent to $v$ if and only if $\mu(x)\in \mathcal{P}(v)$ for some $x\in \mathcal{P}(u)$, allowing parallel edges and self-loops. Any simple $r$-regular graph is equally likely to be chosen under this model.

We study a {\em biased random walk}. On a fixed graph $G$, this process is defined as follows. Initially, all edges are declarded {\em unvisited}, and we choose a vertex $v_0$ uniformly at random as the {\em active} vertex. At any point of the walk, the walk moves from the active vertex $v$ along an edge chosen uniformly at random from the unvisited edges incident to $v$, after which the edge is permanently declared {\em visited}. If there are no unvisited edges incident to $v$, the walk moves along a visited edge chosen uniformly at random. The other endpoint of the chosen edge is declared active, and the process is repeated.

A biased random walk on the random $r$-regular graph can be seen as a random walk on the configuration model, where we expose $\m$ along with the walk as follows. Initially choosing some point $x_0\in \PP$ u.a.r., we walk to $x_1 = \m(x_0)$, chosen u.a.r. from $\PP\setminus\{x_1\}$. Suppose $x_1\in \PP(v_1)$. From $x_1$, the walk moves to some unvisited $x_2 \in \PP(v_1)$. In general, if $W_k = (x_0, x_1,\dots,x_k)$ then (i) if $k$ is odd, the walk moves to $x_{k + 1} = \m(x_k)$ (chosen u.a.r. from $\PP \setminus \{x_0,\dots,x_k\}$ if $x_k$ is previously unvisited), and (ii) if $k$ is even, the walk moves from $x_k \in \PP(v_k)$ to $x_{k + 1}\in \PP(v_k)$, chosen u.a.r. from the unvisited points of $\PP(v_k)$ if such exist, otherwise chosen u.a.r. from all of $\PP(v_k)$.

We define $C(t)$ to be the number of steps taken immediately before the walk exposes its $t$th distinct edge. To be precise, if $W_k = (x_0,\dots,x_k)$ denotes the walk after $k$ steps, then
$$
C(t) = \min \{k : |\{x_0, x_1,\dots,x_k\}| = 2t-1\}.
$$
We also let $W(t) = W_{C(t)}$. Note that $C(t)$ is a random variable over the combined probability space of random graphs and random walks, as opposed to $C_V^b(G_r)$ and $C_E^b(G_r)$ which are variables over the space of random graphs only. We will show (Lemma \ref{lem:phaseone}) that if $t_1 = (1-\log^{-1/2}n)\frac{rn}{2}$ then
$$
\E{C(t_1)} = o(n\log n),
$$
which does not contribute significantly to the cover time. The main part of the proof is calculating $\E{C(t + 1) - C(t)}$ when $t \geq t_1$. We define the random graph $G(t) \subseteq G_r$ as the graph spanned by the first $t$ distinct edges visited by the walk. If, immediately after discovering its $t$th edge, the biased random walk inhabits a vertex incident to no unvisited edges, then a simple random walk commences on $G(t)$, and $C(t + 1) - C(t)$ is the number of steps taken for this random walk to hit a vertex incident to an unvisited edge.

We construct from $G(t)$ a graph $G^*(t)$ by contracting all vertices incident to at least one unvisited edge into one ``supervertex'' $x$. Thus, conditioning on $W(t)$, the graph $G^*(t)$ is a fixed graph, i.e. one with no random edges. We will show that when $t\geq t_1$, w.h.p. $x$ lies on ``few'' cycles of ``short'' length and has the appropriate number of self-loops (to be made precise in Section \ref{sec:randomwalks}), which will imply that the expected hitting time of $x$ for a simple random walk on $G^*(t)$ is
$$
\E{H(x)} \sim \frac{1}{r-2} \frac{rn}{rn-2t}.
$$

Readers familiar with the proof for the cubic graph \cite{CFJ} will recognize the general idea of this outline. In the case of cubic graphs, the set of vertices visited exactly once coincides with the set of vertices incident to one univisited edge, modulo the starting vertex of the walk. This is no longer true when $r \geq 5$, which forces a more detailed study of the edges not visited by the walk.

The paper is laid out as follows. Sections \ref{sec:Grproperties}, \ref{sec:randomwalks} and \ref{sec:uniform} respectively discuss properties of the random regular graph, hitting times of simple random walks, and a uniformity lemma for biased random walks, and may be read in any order. Section \ref{sec:covertime} contains the calculation of the cover time. Sections \ref{sec:setsizes} and \ref{sec:greenedges} are devoted to bounding the sizes of certain sets appearing in the calculations.

\section{Properties of $G_r$}\label{sec:Grproperties}
 Here we collect some properties of random $r$-regular graphs, chosen according to the configuration model. 
\begin{lemma}\label{lem:Grproperties}
Let $r\geq 3$. Let $G_r$ denote the random $r$-regular graph on vertex set $[n]$, chosen according to the configuration model. Let $\om$ tend to infinity arbitrarily slowly with $n$. Its value will always be small enough so that where necessary, it is dominated by other quantities that also go to infinity with $n$.
\begin{enumerate}[(i)]
\item With high probability, the second largest in absolute value of the eigenvalues of the transition matrix for a simple random walk on $G_r$ is at most $0.99$.
\item With high probability, $G_r$ contains at most $\om r^\om$ cycles of length at most $\om$,
\item The probability that $G_r$ is simple is $\OM(1)$.
\end{enumerate}
\end{lemma}
Friedman \cite{JF} showed that for any $\e > 0$, the second eigenvalue of the transition matrix is at most $2\sqrt{r-1}/r + \e$ w.h.p., which gives (i).. Property (ii) follows from the Markov inequality, given that the expected number of cycles of length $k\leq \om$ can be bounded by $O(r^k)$. For the proof of (iii) see Frieze and Karo\'nski \cite{FK}, Theorem 10.3. Note that (iii) implies that any property which holds w.h.p. for the configuration multigraph holds w.h.p. for simple $r$-regular graphs chosen uniformly at random.

Let $G(t)$ denote the random graph formed by the edges visited by $W(t)$. Let $X_i(t)$ denote the set of vertices incident to $i$ red edges in $G(t)$ for $i=0,1,2,3$. Let $\ol X(t) = X_1(t)\cup \dots \cup \cup X_r(t)$ denote the set of vertices incident to at least one red edge. Let $G^*(t)$ denote the graph obtained from $G(t)$ by contracting the set $\ol X(t)$ into a single vertex, retaining all edges. Define $\l^*(t)$ to be the second largest eigenvalue of the transition matrix for a simple random walk on $G^*(t)$. 

We note that by \cite[Corollary 3.27]{af}, if $\G$ is a graph obtained from $G$ by contracting a set of vertices, retaining all edges, then $\l(\G) \leq \l(G)$. This implies that $\l^*(t) = \l(G^*(t)) \leq \l(G) \leq 0.99$ for all $t$. Initially, for small $t$, we find that w.h.p. $G^*(t)$ consists of a single vertex. In this case there is no second eigenvalue and we take $\l^*(t)=0$. This is in line with the fact that a random walk on a one vertex graph is always in the steady state.

We define $C(t)$ to be the number of steps the biased random walk takes to traverse $t$ distinct edges of $G_r$. Of course, if $G_r$ is disconnected and the random walk starts in a connected component of less than $t$ edges, then $C(t) = \infty$. We resolve this by defining a stopping time $T^* = \min\{t : \l^*(t) > 0.99\}$, and setting $C^*(t) = C(\min\{t, T\})$. Strictly speaking, the estimates of $C(t)$ in the upcoming sections are estimates of $C^*(t)$, but we do not make any explicit distinction between the two, noting that by Lemma \ref{lem:Grproperties} (i), w.h.p. $T^* = \infty$ which implies that $C^*(t) = C(t)$ for all $t$.

\section{Simple random walks}\label{sec:randomwalks}

We are interested in calculating $C(t+1) - C(t)$, i.e. the time taken between discovering the $t$th and the $(t+1)$th edge. Between the two discoveries, the biased random walk can be coupled to a simple random walk on the graph induced by $W(t)$, and in this section we derive the hitting time of a certain type of expanding vertex set.

Implicitly, when we state results about a vertex set $S$ in a graph $G$ we are considering a sequence of vertex sets $(S_n)$ in a graph sequence $(G_n)$. We say that $G$ has {\em positive eigenvalue gap} if the second largest eigenvalue $\l_n$ of the transition matrix for $G_n$ satisfies $\lim\sup \l_n < 1$.

Consider a simple random walk on an $r$-regular graph $G = (V, E)$ with eigenvalue gap $1 - \l > 0$. For a set $S$ of vertices and a probability measure $\r$ on $V$, let $\Ex{\r}{H(S)}$ denote the expected hitting time of the set $S$, i.e. the number of steps the walk takes until it reaches $S$, when the initial vertex is chosen according to $\r$. Let $\pi$ denote the stationary distribution of the random walk, uniform in the case of a regular graph and proportional to degrees in general. Let $P_u^{(t)}(v)$ denote the probability that a simple random walk starting at $u$ occupies vertex $v$ at time $t$.
\begin{lemma}\label{lem:hitting}
Suppose $v$ is a vertex of a graph. Then the hitting time of $v$, starting from the stationary distribution $\pi$, is given by
$$
\Ex{\pi}{H(v)} = \frac{Z_{vv}}{\pi_v}
$$
where
$$
Z_{vv} = \sum_{t\geq 0} (P_v^{(t)}(v) - \pi_v).
$$
\end{lemma}
Lemma \ref{lem:hitting} can be found in \cite{af} (Lemma 2.11), and can be applied to hitting times of sets by contracting a set of vertices to a single vertex, retaining all edges. The following bound will be frequently used. Suppose $G$ is a graph with eigenvalue gap $1-\l(G)$, and $S$ is a set of vertices in $G$. Then if $G_S$ is the graph obtained by contracting $S$ into a single vertex, retaining all edges, we have equal hitting times for $S$ in $G$ and $G_S$ and
\begin{align}
\Ex{\pi}{H(S)} = \frac{n}{|S|} \sum_{t\geq 0} \left(P_{S}^{(t)}(S) - \pi_S\right) & \leq \frac{n}{|S|}\sum_{t\geq 0} \l(G_S)^t \\
& = \frac{1}{1-\l(G_S)} \frac{n}{|S|} \leq\frac{1}{1-\l(G)}\frac{n}{|S|}.\label{eq:js}
\end{align}
Indeed, $|P_u^{(t)}(v) -\pi_v| \leq (d(v)/d(u))^{1/2} \l^t$ for any $u,v,t$ in a graph with eigenvalue gap $1-\l$ (see for example Jerrum and Sinclair \cite{JS}), and $\l(\G) \leq \l(G)$ for any $\G$ obtained from $G$ by contracting a set of vertices (see \cite[Corollary 3.27]{af}).

In the following lemma we define $N_d(S)$ to be the set of vertices at distance exactly $d$ from the vertex set $S$. 
\begin{lemma}\label{lem:treehitting}
Suppose $G$ is an $r$-regular graph on $n$ vertices with positive eigenvalue gap. Let $\om$ tend to infinity arbitrarily slowly with $n$. Suppose $S$ is a set of vertices with $|S|$ even such that
$$
|N_d(S)| = (r-1)^d|S|
$$
for all $1\leq d\leq \om$. Then
$$
\Ex{\pi}{H(S)} \sim \frac{r}{r-2}\frac{n}{|S|}.
$$
\end{lemma}

\begin{proof}
We first note that the set $S$ contains exactly $|S|/2$ edges. Indeed, as $|N(S)| = (r-1)|S|$ and the total degree of $S$ is $r|S|$, $S$ contains at most $|S|/2$ edges. As $|N_2(S)| = (r-1)^2|S|$, each vertex of $N(S)$ must have exactly one edge to $S$, implying that $S$ contains at least $|S|/2$ edges.

Consider the graph $G_S$ obtained by contracting $S$ into a single node $s$, retaining all edges. In the graph $G_S$, $s$ has degree $r|S|$.  Then $s$ is a node with exactly $|S|/2$ self-loops, and is otherwise contained in no cycle of length at most $\om$, as $|N_d(S)| = (r-1)^d|S|$ ensures that $G_S$ is locally a tree up to distance $\om$ from $s$. Since $\pi_s = |S| /n = o(1)$ we may choose $\om$ tending to infinity with $\om\pi_s = o(1)$. We have
$$
Z_{ss} = \sum_{t\geq 0} (P_s^{(t)}(s) - \pi_s) = \left[\sum_{t =0}^\om P_s^{(t)}(s)\right] - o(1) + \sum_{t > \om} |P_s^{(t)}(s) - \pi_s|.
$$
Repeating the argument following \eqref{eq:js},
$$
\sum_{t > \om} |P_s^{(t)}(s) - \pi_s| \leq \sum_{t>\om} \l^t = O(\l^\om) = o(1).
$$
We now argue that
$$
\sum_{t = 0}^\om P_s^{(t)}(s) = \frac{r}{r-2} + o(1).
$$
It is argued in Cooper and Frieze \cite{CFreg}, Lemma 7, that with no loop at vertex $s$, the expected number of returns to $s$ within $\om$ steps is $(r-1)/(r-2)+o(1)$. With $|S| / 2$ loops, when at $s$, there is a $1/r$ chance of using the loop and so each visit to $s$ yields $r/(r-1)$ expected returns. I.e. the 2 of \cite{CFreg} become $(r-1)/(r-2)\times r/(r-1)$.
\end{proof}

\begin{definition}\label{def:rootset}
Let $G = (V, E)$ be an $r$-regular graph. A set $S\subseteq V$ is a {\em root set of order $\ell$} if (i) $|S| \geq \ell^5$, (ii) the number of edges with both endpoints in $S$ is between $|S|/2$ and $(1/2 + \ell^{-3})|S|$, and (iii) there are at most $|S| /\ell^3$ paths of length at most $k$ between vertices of $S$ which use no edges fully contained in $S$.
\end{definition}

The set $S$ in Lemma \ref{lem:treehitting} can be thought of as a ``pure'' root set, and we now show that the hitting time remains similar when we allow some impurities.

\begin{lemma}\label{roothitting}
Let $\om$ tend to infinity arbitrarily slowly with $n$. Suppose $G$ is an $r$-regular graph on $n$ vertices whose transition matrix has second largest eigenvalue $\l \leq 0.99$, containing at most $\om r^\om$ cycles of length at most $\om$. If $S$ is a root set of order $\om$ and a simple random walk is initiated at a uniformly random vertex of $G$, then the expected number of steps needed to reach $S$ is
$$
\E{H(S)} \sim \frac{r}{r-2}\frac{n}{|S|}.
$$
\end{lemma}

\begin{proof}
Consider the contracted graph $G_S$, and let $s$ denote the contracted node. Then $s$ has degree $r|S|$, and $s$ has at most $(1/2 +\om^{-3})|S|$ self-loops. Apart from the self-loops, $s$ lies on at most $|S|/\om^3$ cycles of length at most $\om$, as any cycle of $G_S$ containing $s$ corresponds to a path between members of $S$ in $G$.

Let $R = N(S)$, and note that $|R| = \OM(|S|)$. Consider the graph $\G$, defined as $G_S$ induced on the set of vertices at distance $1,2,\dots,\om$ from $s$. Note that $s$ is not included in $\G$. The graph $\G$ contains all of $R$, and as $s$ lies on at most $|S|/\om^3$ short cycles in $G_S$, the number of components in $\G$ containing more than one member of $R$ is $O(|S|/\om^3) = O(|R|/\om^3)$. As $G$ contains at most $\om r^\om$ short cycles, the number of components of $\G$ containing a cycle is at most $\om r^\om = O(|R| / \om^3)$ (choosing $\om$ tending to infinity sufficiently slowly). This leaves $(1-o(1))|R|$ connected components in $\G$ which are all complete $(r-1)$-ary trees of height $\om$, each rooted at a member of $R$ and containing no other member of $R$. Let $T$ denote the set of vertices on such components. 

Arbitrarily choose $|S|/2$ of the self-loops of $s$ in $G_S$, and designate them as {\em good}. Also say that an edge is good if it has both endpoints in $T\cup \{s\}$. All other edges are {\em bad}.

Consider a simple random walk $Z(\t)$ of length $\om$ on $G_S$, starting at $s$. Let $\BB_\t$ denote the event that $Z(\t)$ traverses a bad edge to reach $Z(\t+1)$. Whenever the walk visits $s$, the probability that it chooses a bad edge is $O(\om^{-3})$. If the walk is inside $T$, there are no bad edges to choose. So for any $\t \geq 0$ we have
\begin{align}
P_s^{(\t)}(s) & = \Prob{Z(\t)  = s\cap\ \bigcap_{r=0}^{\t-1} \overline{\BB_r}} + \Prob{Z(\t) = s \cap \bigcup_{r=0}^{\t-1} \BB_r} \\
& = \Prob{Z(\t) = s\cap \bigcap_{r=0}^{\t-1} \ol{\BB_r}} + O(\om^{-2}).
\end{align}
If $\BB_r$ does not occur for any $r \leq \t-1$, then the walk $(Z(0),\dots,Z(\t-1))$ can be viewed as the same Markov chain as considered in Lemma \ref{lem:treehitting}. So, by Lemma \ref{lem:treehitting},
$$
\sum_{\t=0}^\om P_s^{(\t)}(s) = \frac{r}{r-2} + O(\om^{-1}).
$$
\end{proof}

The next lemma is needed in the study of the sizes of $X_i(t)$, and will be applied with $R = X_1(t), S = \ol X(t)$.
\begin{lemma}\label{lem:blueend}
Let $G$ be an $r$-regular graph with positive eigenvalue gap. Let $R \subseteq S \subseteq V$ be vertex sets. Suppose a simple random walk is initiated at a uniformly random vertex $y \in R$, and ends as soon as it hits $S\setminus \{y\}$. Then there is a constant $B > 0$ such that for any $x\in S$, the probability that the walk ends at $x$ is at most $B / |R|$.
\end{lemma}
\begin{proof}
Let $y\in R$ denote the random starting point of the walk and let $x\in S\setminus\{y\}$. We will in general write $(w_0,w_1,\dots)$ for a simple random walk. Define $S(x, y) = S\setminus\{x, y\}$. The following reversibility property is central to the proof. If $T_{S(x, y)}$ denotes the first time at which $S(x, y)$ is visited, then as the graph is regular,
\begin{equation}
P_x(w_t = y\AND T_{S(x, y)} > t) = P_y(w_t = x\AND T_{S(x, y)} > t), \quad \text{for all $t\geq 0$.}
\end{equation}
Let $\m$ denote the uniform probability measure on $R$. Then, writing $\cdot$ or $y$ for the starting vertex chosen by $\m$, the probability that $x$ is the first vertex of $S\setminus\{y\}$ to be hit is
\begin{align}
P_\m(T_x < T_{S(x, \cdot)}) & = \sum_{t\geq 0} P_\m(t = T_x\AND T_{S(x, \cdot)} > t) \\
& \leq \sum_{t\geq 0} P_\m(w_t = x \AND T_{S(x, \cdot)} > t) \\
& = \frac{1}{|R|} \sum_{t\geq 0}\sum_{y\in R} P_y(w_t = x \AND T_{S(x, y)} > t) \\
& = \frac{1}{|R|} \sum_{y\in R}\sum_{t\geq 0} P_x(w_t = y \AND T_{S(x, y)} > t).
\end{align}
Write
\begin{equation}\label{eq:Qdef}
Q_{xy}(t) = P_x(w_t = y \AND w_1,\dots,w_{t-1}\neq y \AND T_{S(x, y)} > t),
\end{equation}
so that $Q_{xy}(t)$ denotes the probability that the walk avoids $S(x, y)$ and first visits $y$ is at time $t$ (or returns to $y$ if $x = y$). Then for all $t\geq 1$,
$$
P_x(w_t = y \AND T_{S(x, y)} > t) = \sum_{s=1}^t Q_{xy}(s) P_y(w_{t-s} = y\AND T_{S(x, y)} > t-s),
$$
so
$$
\sum_{t\geq 0}P_x(w_t = y\AND T_{S(x, y)} > t) = \left[\sum_{s\geq 1} Q_{xy}(s)\right]\left[\sum_{t\geq 0}P_y(w_t = y \AND T_{S(x, y)} > t)\right].
$$
As hitting times are almost surely finite, writing $T_y^+ = \min\{t \geq 1 : w_t = y\}$, we have
$$
\sum_{s\geq 1} Q_{xy}(s) = P_x(T_y^+ < T_{S(x, y)}).
$$
We have
\begin{align}
\sum_{t\geq 0}P_y(w_t = y\AND T_{S(x, y)} > t) & = 1 + \sum_{t\geq 1}\sum_{s=1}^t Q_{yy}(s)P_y(w_{t-s} = y\AND T_{S(x, y)} > t-s) \\
& = 1 + \left[\sum_{s\geq 1} Q_{yy}(s)\right]\left[\sum_{t\geq 0} P_y(w_t = y\AND T_{S(x, y)} > t)\right],
\end{align}
so as $\sum_{s\geq 1} Q_{yy}(s) = P_y(T_y^+ < T_{S(x, y)})$,
$$
\sum_{t\geq 0} P_y(w_t = y\AND T_{S(x, y)} > t) = \frac{1}{1-\sum_{s\geq 1}Q_{yy}(s)} = \frac{1}{P_y(T_{S(x, y)} < T_y^+)}.
$$
So we have
\begin{equation}\label{eq:almostdone}
P_\m(T_x < T_{S(x, y)}) = \frac{1}{|R|}\sum_{y\in R} \frac{P_x(T_y < T_{S(x, y)})}{P_y(T_{S(x, y)} < T_y^+)}.
\end{equation}
\begin{claim}\label{cl:B}
$$
P_y(T_{S(x, y)} < T_y^+) \geq \frac{1-\l}{2}.
$$
\end{claim}
Applying Claim \ref{cl:B} to \eqref{eq:almostdone} finishes the proof, as
$$
P_\m(T_x < T_S) \leq \frac{2}{(1-\l)|R|} \sum_{y\in R} P_x(T_y < T_{S(x, y)}),
$$
and this is a sum of mutually exclusive events, which therefore sums to at most $1$.
\begin{proof}[Proof of Claim \ref{cl:B}]
Corollary 2.8 of \cite{af} states that for any states $i, j$ in a Markov chain,
$$
P_i(T_j < T_i^+) = \frac{1}{\pi_i(\Ex{i}{T_j} + \Ex{j}{T_i})}.
$$
We consider running a simple random walk on the graph obtained by contracting $S$ to a single vertex, and set $i = y, j = S(x, y)$. As $\pi_y = 1/n$, we have
$$
P_y(T_{S(x, y)} < T_y^+) = \frac{n}{\Ex{y}{T_{S(x,y)}} + \Ex{S(x, y)}{T_y}}.
$$
For any two states $k, \ell$ in a Markov chain with eigenvalue gap $1-\l$ we have
$$
\Ex{k}{T_\ell} \leq O(\log n) + \frac{1}{1-\l} \frac{1}{\pi_\ell},
$$
so
$$
P_y(T_S < T_y^+) \geq \frac{n}{O(\log n) + \frac{1}{1-\l}\frac{n}{|S(x, y)|} + \frac{n}{1-\l}} \geq \frac{1-\l}{2}.
$$
\end{proof}

\end{proof}

\section{The set $\ol X$}
\label{sec:uniform}

The walk $W(t)$ induces a colouring on the edges and vertices of $G_r$ as follows. An edge is coloured red, green or blue if it has been visited zero, one or at least two time(s), respectively. A vertex is (i) green if it is incident to exactly $r-1$ green edges and one red edge, (ii) red if it is incident to red edges only, and (iii) blue otherwise.

Recall that $X_i(t)$ denotes the set of vertices incident to exactly $i$ red edges in $W(t)$. We let $X_1^g(t)$, $X_1^b(t)$ denote the green and blue vertices of $X_1(t)$, respectively, and set
$$
Z(t) = X_1^b(t) \cup \bigcup_{i=2}^r X_i(t).
$$
We have $\ol X(t) = X_1^g(t) \cup Z(t)$. Note that the number of configuration points not visited by $W(t)$ is exactly $rn-2t+1$. We will eventually show that $|X_1^g(t)| = (rn-2t)(1 - o(1))$ and $|Z(t)| = o(rn-2t)$, so that $X_1^g(t)$ makes up almost all of $\ol X(t)$ when $t = (1-\d)\frac{rn}{2}$ for some $\d = o(1)$. In this section we present a ``sprinkling'' tool used to show that $X_1^g(t)$ is a root set of order $\om$, which will imply that $\ol X(t)$ is also a root set of order $\om$. Before this, we state our results on the necessary set sizes.

\subsection{Set sizes}

Define
\begin{equation}\label{eq:didef}
\d_0 = \frac{1}{\log\log n},\ \d_1 = \frac{1}{\log^{1/2} n},\ \d_2 = \frac{1}{\log^2 n},\ \d_3 = n^{-3/4},\ \d_4 = n^{-1}\log n,
\end{equation}
and $t_i = (1-\d_i)\frac{rn}{2}$ for $i=0,1,2,3$. From this point on we will use $t$ and $\d$ interchangeably to denote time, and the two are always related by $t = (1-\d)\frac{rn}{2}$.

\begin{lemma}\label{lem:concentration}
Let $\om$ tend to infinity arbitrarily slowly, and let $0 < \e < r-2$. Then
\begin{enumerate}[(i)]
\item for any fixed $t\geq t_1$, w.h.p. the number of green vertices in $X_1(t)$ is
$$
X_1^g(t) = rn\d(1-o(1)),
$$
\item for any fixed $t_1 \leq t \leq t_3$, w.h.p. the number of green edges satisfies
$$
\F(t) \geq n\d^{1/2},
$$
\item for any fixed $t\geq t_1$, w.h.p. the size of $Z(t)$ is
$$
Z(t) = O(n\d^{3/2}),
$$
\item for any fixed $t \geq t_1$, w.h.p. the number of unvisited vertices is
$$
X_r(t) = n\d^{r/2}(1 + o(1)).
$$
\end{enumerate}
\end{lemma}
Parts (i), (iii) and (iv) are proved in Section \ref{sec:setsizes}, and part (ii) in Section \ref{sec:greenedges}.

\subsection{A sprinkling technique}

The green edges and vertices are the focus of this section. Suppose $W = W(t) = (x_0, x_1,\dots,x_k)$ and that $x_{2i}, x_{2i+1}, x_{2i + 2}, x_{2i+3}$ are consecutive configuration points visited exactly once by $W$, where $x_{2i+1}, x_{2i+2}$ belong to a green vertex $v$. In this situation we call the pair $(x_{2i + 1}, x_{2i+2})$ a {\em green link}. Let $L(W) \subseteq \PP\times \PP$ be the set of green links in $W$, and form the {\em contracted walk} $\langle W\rangle$ by removing all green links from $W$, i.e. replacing the two edges $(x_{2i}, x_{2i+1}), (x_{2i+2}, x_{2i+3})$ by the single edge $(x_{2i}, x_{2i+3})$ for all green links $(x_{2i+1}, x_{2i+2})$. Two walks $W_1, W_2$ are said to be equivalent if $\langle W_1\rangle = \langle W_2\rangle$ and $L(W_1) = L(W_2)$, and we let $[W] = (\langle W\rangle, L)$ denote the equivalence class of the walk $W$.

\begin{lemma}\label{lem:uniform}
If $W$ is such that $\Prob{[W(t)] = [W]} > 0$, then
$$
\Prob{W(t) = W \mid [W(t)] = [W]} = \frac{1}{|[W]|}.
$$
\end{lemma}

\begin{proof}
Let $W$ be a walk with $\Prob{W(t) = W} > 0$. We can calculate the probability of $W(t) = W$ exactly. There are two different types of steps a walk can take. Suppose the walk has visited $t$ distinct edges. If the walk occupies a vertex incident to no red edges, it chooses an edge with probability $r^{-1}$. If the walk occupies a vertex incident to $k$ red edges, it chooses one of the $k$ red edges with probability $k^{-1}$. The other endpoint of the red edge is chosen uniformly at random from $rn-2t-1$ configuration points. So the probability of $W(t) = W$ is
$$
\Prob{W(t) = W} = \frac{1}{rn} \prod_{k = 2}^{r} k^{-i_k} \prod_{s=0}^t \frac{1}{rn-2s-1},
$$
for some integers $i_2,\dots,i_r \geq 0$, counting the number of steps of the different types. The $1/rn$ factor accounts for the starting point of the walk. Now, if $W_1\sim W_2$, then $W_1$ and $W_2$ contain the same number of edges, and $i_k(W_1) = i_k(W_2)$ for $k=2,\dots,r$. Indeed, $W_1$ and $W_2$ only disagree in which order they visit the links in $L$.
\end{proof}

Conditioning on $[W(t)] = (\langle W(t)\rangle, L)$, we can now generate $W(t)$ as follows. Suppose $\langle W(t)\rangle = (x_{\ang 0}, x_{\ang 1}, \dots,x_{\ang k})$ is the contracted walk with $\f$ green edges. Let $F$ denote the set of green edges in $W(t)$. Arbitrarily assigning some order to $L$, let $(p_1, q_1)$ denote the first link, and choose some $(x_{\ang{2i}}, x_{\ang{2i+1}})\in F$ uniformly at random. We reroute the edge $(x_{\ang{2i}}, x_{\ang{2i+1}})$ through $(p_1, q_1)$, forming
$$
W_1 = (x_{\ang 0}, \dots,x_{\ang{2i}}, p_1, q_1, x_{\ang{2i+1}}, \dots,x_{\ang k}).
$$
Form $F_1$ from $F$ by replacing $(x_{\ang{2i}}, x_{\ang{2i+1}})$ by $(x_{\ang{2i}}, p_1)$ and $(q_1, x_{\ang{2i+1}})$. We repeat the above with the next $(p_2, q_2) \in L$, forming $W_2, W_3,\dots$, until all links have been placed in the walk. The final walk is $W(t)$. We refer to this as sprinkling the links into the green edges of the contracted walk, and note that the initial green edges $F$ can be seen as the urns in a P\'olya urn process.

\subsection{Set structure}

\begin{lemma}\label{lem:sprinkling}
Suppose $[W(t)]$ is an equivalence class with $rn-2t+1$ free configuration points, where $rn-2t \to\infty$ and $rn-2t = o(n)$, and suppose $\F(t) \geq n\d^{1/2}$, $X_1^g(t) = rn\d(1-o(1))$ and $Z(t) = O(n\d^{3/2})$. Let $\om$ tend to infinity arbitrarily slowly. If $W(t)$ is chosen uniformly at random from $[W(t)]$, then with high probability, the set $\ol X(t)$ associated with $W(t)$ is a root set of order $\om$.
\end{lemma}

\begin{proof}
Say that $v \in X_1^g$ is {\em bad} if its distance to the rest of $\ol X$ is at most $\om$, or if it lies on a cycle of length at most $\om$. Otherwise it is {\em good}. As $X_1^g$ makes up almost all of $\ol X = X_1^g \cup Z$, to show that $\ol X$ is a root set of order $\om$ it is enough to show that the number of bad vertices in $X_1^g$ is at most $X_1^g / \om^3$.

By Lemma \ref{lem:Grproperties}, at most $\om r^\om = o(X_1^g(t))$ vertices lie on cycles of length at most $\om$. Let $\ell_1, \ell_2\in L$ be links. As they are sprinkled onto the $\f = \OM(n\d^{1/2})$ green edges of $[W(t)]$, the probability that they are placed within distance $\om$ of each other is $O(r^\om / \f(t))$. Indeed, any green edge in $[W(t)]$ is within graph distance at most $\om$ of at most $r^\om$ other green edges. The expected number of pairs $\ell_1, \ell_2$ within distance $\om$ of each other is bounded by
\begin{equation}\label{eq:linkunionbound}
\sum_{\ell_1 \neq \ell_2} O\bfrac{r^\om}{\f(t)} = O\bfrac{L(t)^2r^\om}{\f(t)} = O\bfrac{n^2 \d^2 r^\om}{n\d^{1/2}} = o(n\d).
\end{equation}
As $X_1^g = \OM(n\d)$, choosing $\om$ small enough this shows that all but $\ol X(t)/\om^3$ of the links are at distance at least $\om$ from other links. We conclude that $\ol X(t)$ is a root set of order $\om$.
\end{proof}

\section{Calculating the cover time}\label{sec:covertime}

Recall the definitions \eqref{eq:didef} of $\d_i$ and $t_i$, $0\leq i\leq4$. We begin by showing that the time taken to find the first $t_1$ edges contributes insignificantly to the cover time.
\begin{lemma}\label{lem:phaseone}
$$
\E{C(t_1)} = o(n\log n).
$$
\end{lemma}
This is proved in Section \ref{sec:phaseone}. We then move on to estimating the expected cover time increment for larger $t$.
\begin{lemma}\label{lem:phasetwothree}
For $t_1 \leq t \leq t_4$ and any $\e > 0$,
$$
\E{C(t + 1) - C(t)} = \left(\frac{r}{r-2} \pm \e\right) \frac{n}{rn-2t}.
$$
\end{lemma}
The time to discover the final $O(\log n)$ edges can be bounded using \eqref{eq:js}:
$$
\E{C\bfrac{rn}{2} - C(t_4)} \leq \sum_{t=t_4}^{rn/2-1} O\bfrac{n}{rn-2s} = o(n\log n).
$$
This shows that for $t \geq t_1$,
\begin{multline}
\E{C(t)} = \E{C(t_1)} + \sum_{s = t_1}^{t - 1} \E{C(s + 1) - C(s)} \\ 
= \frac{r\pm \e}{2(r-2)}n\log\bfrac{rn}{rn-2t+1} + o(n\log n),
\end{multline}
proving the edge cover time statement of Theorem \ref{thm:exptheorem}. The vertex cover time follows from a simple argument using Lemma \ref{lem:concentration} (iv): the walk is expected to cover all but $s$ vertices at time $\d = ((s + 1)/n)^{2/r}$, which accounts for the $r/2$ factor separating the vertex and edge cover times. The full calculation is carried out for $r=3$ in \cite{CFJ}, and generalizing this to larger $r$ is trivial.

To prove Theorem \ref{thm:whptheorem} it remains to show that $\Ex{G}{C(t)} \sim \E{C(t)}$ for almost all fixed $r$-regular graphs $G$. Again, this calculation is essentially identical to that for $r = 3$ in \cite{CFJ}, and we exclude the rather lengthy calculations from this paper. Note that the proof is valid only for $t \geq t_2$, which is why the range $t_1 \leq t < t_2$ is excluded from the statement of Theorem \ref{thm:exptheorem}.

\subsection{Phase one: Proof of Lemma \ref{lem:phaseone}}\label{sec:phaseone}
With $t_1$ as in \eqref{eq:didef}, we show that $\E{C(t_1)} = o(n\log n)$. Suppose $W(t) = (x_0,x_2,\dots,x_k)$ for some $t, k$. If $x_k\in \PP(\ol X(t))$ then $x_{k + 1} = \m(x_{k})$ is uniformly random inside $\PP(\ol X(t))\setminus \{x_{k}\}$, and since $C(t+1) = C(t) + 1$ in the event of $x_{k + 1}\in \PP(X_2\cup \dots \cup X_r)$, we have
\begin{equation}\label{eq:conditionincrement}
\E{C(t+1) - C(t)} \leq 1 + \E{C(t+1)-C(t) \mid x_{k + 1}\in \PP(X_1)}\Prob{x_{k + 1}\in \PP(X_1)},
\end{equation}
We use the following theorem of Ajtai, Koml\'os and Szemer\'edi \cite{AKS} to bound the expected change when $x_{k + 1}\in \PP(X_1)$.
\begin{theorem}\label{thm:aks}
Let $G = (V, E)$ be an $r$-regular graph on $n$ vertices, and suppose that each of the eigenvalues of the adjacency matrix with the exception of the first eigenvalue are at most $\l_G$ (in absolute value). Let $A$ be a set of $cn$ vertices of $G$. Then for every $\ell$, the number of walks of length $\ell$ in $G$ which avoid $A$ does not exceed $(1-c)n((1-c)r + c\l_G)^\ell$.
\end{theorem}

The set $A$ of Theorem \ref{thm:aks} is  fixed. In our case we choose a point $x_{k + 1}$ uniformly at random from $\PP(X_1(t))$, so we consider a simple random walk initiated at a uniformly random vertex $u \in X_1(t)$. The subsequent walk  now begins at vertex $u$ and continues until it hits a vertex of $Y_u=\ol X(t)\setminus\{u\}$. Because the vertex $u$ is random, the set $Y_u$ differs for each possible exit vertex $u \in X_1(t)$. To apply Theorem \ref{thm:aks}, we  split $X_1(t)$ into two disjoint sets $A, A'$ of (almost) equal size. For  $u \in A$, instead of considering the number of steps needed to hit $Y_u$, we can upper bound this by the number of steps needed to hit $B' = A'\cup X_2\cup\dots\cup X_r$, and vice versa. Suppose without loss of generality that $u\in A$.

Let $S(\ell)$ be a simple random walk of length $\ell$ starting from a uniformly chosen vertex of $A$. Thus $S(\ell)$ could be any of $|A| r^\ell$ uniformly chosen random walks. Let $c = |B'| / n$. The probability $p_\ell$ that a randomly chosen walk of length $\ell$ starting from $A$ has avoided $B'$ is, by Theorem \ref{thm:aks}, at most
$$
p_\ell \leq \frac{1}{(|X_1(t)|/2)r^\ell}(1-c)n(r(1-c) + c\l_G)^\ell \leq \frac{2(1-c)n}{|X_1(t)|}((1-c) + c\l)^\ell,
$$
where $\l\leq.99$ (see Lemma \ref{lem:Grproperties}) is the absolute value of the second largest eigenvalue of the transition matrix of $S$. Thus
\beq{AHC}{
\Ex{A}{H(C)} \leq \sum_{\ell\geq 1} p_\ell \leq \frac{2(1-c)n}{|X_1(t)|} \frac{1}{c(1-\l)}.
}
So,
\begin{equation}\label{eq:aksincrement}
\E{C(t+1) - C(t) \mid x_{2k}\in \PP(X_1(t))} = O\bfrac{(n - |B'|)n}{|X_1||B'|}.
\end{equation}
Now, for any $t$ we have $r^{-1}(rn-2t) \leq |B'| \leq rn-2t$, so summing over $0\leq t \leq t_1$, \eqref{eq:conditionincrement} gives $\E{C(t_1)} = o(n\log n)$.

\subsection{Phase two: Proof of Lemma \ref{lem:phasetwothree}}\label{sec:phasetwo}

\begin{lemma}
Let $\e > 0$. For $t_1 \leq t \leq t_4$,
$$
\E{C(t + 1) - C(t)} = \left(\frac{r}{r-2} \pm \e\right) \frac{n}{rn-2t}.
$$
\end{lemma}

\begin{proof}
The proof of Lemma \ref{lem:phasetwothree} is based on the following calculation. Define events
\begin{align}
\AA(t) & = \left\{|X_1^g(t) - (rn-2t)| \leq \frac{rn-2t}{\om}\right\}, \\
\BB(t) & = \{\ol X(t) \text{ is a root set of order } \om\},
\end{align}
and set $\EE(t) = \AA(t) \cap \BB(t)$. Then for any $\e > 0$, $\E{C(t + 1) - C(t)}$ can be calculated as
$$
\left(\frac{r}{r-2} \pm \e\right)\frac{n}{rn-2t}\Prob{\EE(t)} + O\left(\frac{n}{rn-2t} \Prob{\ol{\EE(t)}}\right) + O(\log n).
$$
Indeed, suppose $\EE(t)$ holds. As $X_1(t)$ contains almost all unvisited configuration points, edge $t$ is attached to some $v\in X_1(t)$ w.h.p., and a simple random walk commences at $v$, ending once it hits $\ol X \setminus\{v\}$. As the vertices of $\ol X$ are spread far apart, it is unlikely that this happens within $O(\log n)$ steps. After a logarithmic number of steps, the random walk has mixed to within $\e$ of the stationary distribution $\pi$ in total variation. Lemma \ref{roothitting} shows that after this point, the expected time taken to hit $\ol X$ is $(r/(r-2)\pm \e) n/|\ol X|$, and as $\AA(t)$ holds we have $|\ol X| \sim (rn-2t)$. If $\EE(t)$ does not hold, then we use the bound \eqref{eq:js}, stating that the hitting time is $O(n/|\ol X|) = O(n/(rn-2t))$ (as $|\ol X| \geq (rn-2t)/r$) as long as the graph has a positive eigenvalue gap. We refer to the discussion in Section \ref{sec:Grproperties} justifying our assumption that the second largest eigenvalue stays at most $0.99$ throughout the process.

Lemmas \ref{lem:concentration} and \ref{lem:sprinkling} show that indeed, $\Prob{\EE(t)} = 1-o(1)$ for $t_1 \leq t \leq t_3$. It remains to show that $\Prob{\EE(t)} = 1-o(1)$ for $t_3\leq t\leq t_4$. Fix $t > t_3$. As $Z(t_3) = O(n\d_3^{3/2}) = o(1)$, we have $Z(t) \subseteq Z(t_3) = \emptyset$ w.h.p. by Markov's inequality. Note that this implies $X_1^g(t) \subseteq X_1^g(t_3)$ and $|X_1^g(t)| = rn-2t$. We also have
$$
\frac{L(t_3)^2}{\F(t_3)} = O\bfrac{n^2\d_3^2}{n\d_3^{1/2}} = O\left(n\left(n^{-3/4}\right)^{3/2}\right) = o(1).
$$
Thus, repeating the calculation \eqref{eq:linkunionbound} of Lemma \ref{lem:sprinkling}, we have that no two vertices of $X_1^g(t_3)$ are placed within distance $\om$ of each other. As $X_1^g(t) \subseteq X_1^g(t_3)$, the same must be true of $X_1^g(t)$. Thus the only vertices of $X_1^g(t)$ which violate the root set constraints are those placed on the $\om r^\om$ short cycles of $G_r$, and choosing $\om$ small enough we have $\om r^\om \leq (rn-2t)/\om^3$ for all $t_3\leq t\leq t_4$, so w.h.p. $\ol X(t) = X_1^g(t)$ is a root set of order $\om$.
\end{proof}

\section{Set sizes}\label{sec:setsizes}

Recall the definition
$$
Z(t) = X_1^b(t) \cup \bigcup_{i=2}^r X_i(t),
$$
where $X_i$ denotes the set of vertices incident to $i$ unvisited edges, and $X_1^b$ is the set of vertices in $X_1$ which are incident to at least one edge which has been visited more than once.
\begin{lemma}\label{lem:Ztbound}
There exists a constant $B > 0$ such that for $t \geq t_0$ and $0 < \th = o(1)$,
$$
\E{e^{\th Z(t)}} \leq \exp\left\{\th Bn\d^{3/2}\right\}.
$$
\end{lemma}

\begin{proof}
We show that there exists a $B > 0$ such that for any $m\geq 1$,
$$
\Prob{[m] \subseteq Z(t)} \leq (B\d)^{3m/2},
$$
beginning with $m = 1$ before the general statement. Let $\LL = \LL(r)$ denote the set of vectors $(\ell_1,\ell_2,\dots,\ell_k)$ with $\ell_i \in \{1,2\}$ such that $\sum \ell_i \leq r-1$, including in $\LL$ the empty vector $\emptyset$, excluding the vector $(2,2,\dots,2)$ consisting of $(r-1)/2$ copies of $2$ (which corresponds to $X_1^g$, as we will see). The vector $\ell = (\ell_1,\dots,\ell_k)$ counts the number of new configuration points of a vertex $v$ that are used the first $k$ times $v$ is visited. We partition
$$
Z(t) = \bigcup_{\ell\in \LL} Z_\ell(t),
$$
where $v \in Z_\ell(t)$ for $\ell = (\ell_1,\dots,\ell_k)$ if and only if there exists a sequence $0 < s_1 < s_2 < \dots < s_k \leq t$ such that $v$ moves from $X_{r-\ell_1-\dots-\ell_{j-1}}$ to $X_{r-\ell_1-\dots-\ell_j}$ at time $s_j$ for $j=1,\dots,k$, and is in $X_{r-\ell_1-\dots-\ell_k}$ at time $t$. If $v\in X_i$ at time $s$, the probability that $v$ is chosen by random assignment is $i /(rn-2s)$, while Lemma \ref{lem:blueend} shows that the probability that $v$ is at the end of a blue walk is $O(1 / (rn-2s))$. In either case, the probability that $v$ moves from one set to another is at most $B/(rn-2s)$ for some $B > 0$. For a fixed $\ell = (\ell_1,\dots,\ell_k)\in \LL$, with $s_0 = 1$,
\begin{align}
\Prob{1\in Z_\ell(t)} \leq & \sum_{s_1 < \dots < s_k}\prod_{j=1}^{k} \left[\prod_{s = s_{j-1} + 1}^{s_j - 1} \left(1 - \frac{r - (\ell_1 + \dots + \ell_{j-1})}{rn-2s}\right) \frac{B}{rn-2s_j}\right] \\
& \times \prod_{s = s_k + 1}^t \left(1 - \frac{r - (\ell_1 + \dots + \ell_k)}{rn-2s}\right). \label{eq:1Zt}
\end{align}
For $b\geq 1$ we use the bound
\begin{align}
\prod_{s = t_0}^{t} \left(1 - \frac{b}{rn-2s}\right)
& \leq \exp\left\{-\frac{b}{2} \sum_{s = t_0}^t \frac{1}{\frac{rn}{2} - s}\right\} \\
& \leq \exp\left\{-\frac{b}{2} \int_{t_0}^{t} \frac{dx}{\frac{rn}{2} - x}\right\} \\
& = \bfrac{rn-2t}{rn-2t_0}^{b/2}. \label{eq:exbound}
\end{align}
Combining \eqref{eq:1Zt} and \eqref{eq:exbound}, the probability that $1\in Z_\ell(t)$ is bounded above by
\begin{align}\label{eq:1Ztb}
\sum_{s_1 < \dots < s_k} \left[\prod_{j=1}^k \frac{B}{rn-2s_j}\bfrac{rn-2s_j}{rn-2s_{j-1}}^{(r-(\ell_1 + \dots + \ell_{j-1}))/2}\right] \bfrac{rn-2t}{rn-2s_k}^{(r-(\ell_1 + \dots + \ell_k))/2}.
\end{align}
Collecting powers of $rn-2s_j$ for $j=1,\dots,k$, we have
$$
\Prob{1\in Z_\ell(t)} \leq B^k \frac{(rn-2t)^{(r - (\ell_1 + \dots + \ell_k)) / 2}}{(rn)^{r/2}} \sum_{s_1 < \dots < s_k} \prod_{j=1}^k (rn-2s_j)^{\ell_j/2 - 1}.
$$
Let $N$ denote the number of indices $j\in\{1,\dots,k\}$ with $\ell_j = 1$. Then
$$
\sum_{s_1 < \dots < s_k} \prod_{j=1}^k (rn-2s_j)^{\ell_j/2 - 1} \leq \prod_{j=1}^k \left(\sum_{s = 0}^t (rn-2s_j)^{\ell_j/2-1}\right) \leq n^{k - N} (rn-2t)^{N/2},
$$
which implies that
$$
\Prob{1\in Z_\ell(t)} \leq \frac{B^k}{r^{r/2}} (rn-2t)^{(r + N - (\ell_1 + \dots + \ell_k)) / 2} n^{k-N-r/2}.
$$
As $\ell_1 + \dots + \ell_k = 2k - N$, we have $(r + N - (\ell_1 + \dots + \ell_k)) / 2 = r/2 - k + N$. So
$$
\Prob{1\in Z_\ell(t)} \leq \frac{B^k}{r^{k-N}} \d^{r/2 -k + N}.
$$
We now argue that $r/2-k + N\geq 3/2$, or equivalently $2(k-N) \leq r-3$, for all $\ell\in \LL$. Firstly, if $\ell_1 + \dots + \ell_k \leq r-3$ then we have $2(k-N) \leq 2k-N = \ell_1 + \dots + \ell_k \leq r-3$. Secondly, if $\ell_1 + \dots + \ell_k = r-2$ then as $r-2$ is odd we have $N \geq 1$, so $2(k-N) \leq 2k-N - 1 \leq r-3$. Finally, if $\ell_1 + \dots + \ell_k = r-1$ then (as $(2,2,\dots,2)\notin \LL$) we have $N\geq 2$, so $2(k-N) \leq 2k-N-2 \leq r-3$.

As $|\LL(r)|$ is a function of $r$ only, and therefore constant with respect to $n$, it follows that
$$
\Prob{1\in Z(t)} = \sum_{\ell\in\LL(r)} \Prob{1\in Z_\ell(t)} = O(\d^{3/2}).
$$

We turn to bounding the probability that $[m]\subseteq Z(t)$. We fix $\ell^{(1)},\dots,\ell^{(m)} \in \LL$ and bound the probability that $i \in Z_{\ell^{(i)}}(t)$ for $i=1,\dots,m$. Let $k(i) = \dim \ell^{(i)}$ denote the number of components of $\ell^{(i)}$. Then, summing over all choices $s_j^{(i)}$ for $1\leq i\leq m$ and $1\leq j\leq k(i)$,
\begin{align}
 & \Prob{i\in Z_{\ell^{(i)}}(t), i = 1,\dots,m} \\
& \leq \sum_{s_j^{(i)}}\prod_{i=1}^m B^{k(i)}\frac{(rn-2t)^{(r-\sum_j \ell_j^{(i)})/2}}{(rn)^{r/2}}\prod_{j=1}^{k(i)} (rn-2s_j^{(i)})^{\ell^{(i)}_j / 2 - 1} \\
& \leq \prod_{i=1}^m \left[B^{k(i)}\frac{(rn-2t)^{(r-\sum_j \ell_j^{(i)})/2}}{(rn)^{r/2}}\prod_{j=1}^{k(i)} \left(\sum_{s=0}^t (rn-2s)^{\ell_j^{(i)}/2 - 1}\right)\right] \\
& \leq B^{\sum k(i)} \d^{3m/2} = O((B^r\d)^{3m/2}).
\end{align}
Summing over all $O(m)$ choices of $\ell^{(i)}, i = 1,\dots,m$, we have
$$
\Prob{[m] \subseteq Z(t)} = O(m(B^r\d)^{3m/2}) \leq (C\d)^{3m/2}
$$
for some constant $C > 0$. By symmetry the same bound holds for any vertex set of size $m$. It follows that for any $m$, writing $(n)_m = n(n-1)\dots(n-m+1)$,
$$
\E{(Z(t))_m} \leq (n)_m \times (C\d)^{3m/2} \leq (Cn\d^{3/2})^m.
$$
For $s > 1$ we apply the binomial theorem to obtain
$$
\E{s^{Z(t)}} = \E{(1 + (s-1))^{Z(t)}} = \sum_{m\geq 0} \frac{\E{(Z(t))_m}(s-1)^m}{m!}.
$$
We set $s = e^\th \leq 1 + 2\th$ (as $\th = o(1)$) to obtain
$$
\E{e^{\th Z(t)}} \leq \sum_{m\geq 0} \frac{(Cn\d^{3/2})^m (2\th)^m}{m!} \leq \exp\left\{\th D n\d^{3/2}\right\},
$$
for some $D > 0$.
\end{proof}

\begin{corollary}\label{cor:X1g}
For $t = (1 - \d)\frac{rn}{2}$ with $\d = o(1)$, and $0 < \th = o(1)$,
$$
\E{e^{-\th X_1^g(t)}} = \exp\left\{-\th rn\d(1 - o(1))\right\}
$$
\end{corollary}
\begin{proof}
The number of free configuration points at time $t$ is $rn-2t$, so
$$
rn-2t = \sum_{i=1}^r iX_i(t) \leq X_1^g(t) + rZ(t).
$$
By Lemma \ref{lem:Ztbound} we have
$$
\E{e^{-\th X_1^g(t)}} \leq e^{-\th(rn-2t)}\E{e^{r\th Z(t)}} = \exp\left\{-\th rn\d(1 -o(1))\right\}.
$$
\end{proof}

The technique used to prove Lemma \ref{lem:Ztbound} can be strengthened to obtain concentration for the number of unvisited vertices $X_r(t)$.
\begin{lemma}\label{lem:Xr}
For $\th > 0$,
\begin{equation}\label{eq:Xrmgf}
\E{e^{\th X_r(t)}} \leq \exp\left\{2\th n\d^{r/2}\right\}.
\end{equation}
Furthermore, if $t = (1-\d)\frac{rn}{2}$ with $\d = o(1)$ and $n\d^{r/2}\to\infty$, then for any $\om$ tending to infinity arbitrarily slowly,
$$
\Prob{|X_r(t) - n\d^{r/2}| > \frac{n\d^{r/2}}{\om^{1/2}}} \leq \frac{1}{\om}.
$$
Finally, if $n\d^{r/2} = o(1)$ then $X_r(t) = 0$ w.h.p.
\end{lemma}
\begin{proof}
From \eqref{eq:exbound} we have for any $m$,
\begin{equation}\label{eq:mbound}
\Prob{[m] \subseteq X_r(t)} = \left(1 - \frac{m}{n}\right) \prod_{s=0}^t \left(1 - \frac{rm}{rn-2s-1}\right) \leq \d^{rm/2}.
\end{equation}
The inequality \eqref{eq:Xrmgf} follows from the arguments used in the proof of Lemma \ref{lem:Ztbound}.

For $m = 1$ we need the converse inequality to \eqref{eq:exbound}. From $\ln(1-x) = -x - x^2$ we have
\begin{align}
\prod_{s = 0}^t \left(1 - \frac{r}{rn-2s}\right) & = \exp\left\{-\sum_{s=0}^t\left[ \frac{r}{rn-2s} + \frac{r^2}{(rn-2s)^2}\right]\right\} \\
& \geq \exp\left\{-\int_{-1}^t \left[\frac{r}{rn-2x} + \frac{r^2}{(rn-2x)^2}\right]dx\right\} \\
& \geq \bfrac{rn-2t}{rn}^{r/2}(1 - o(1)).
\end{align}
Together with \eqref{eq:exbound} this shows that $\E{X_r(t)} = n\d^{r/2}(1 - o(1))$. From \eqref{eq:mbound} we have $\E{(X_r(t))_2} \leq n(n-1)\d^r$. We conclude that the leading terms of $\E{(X_r(t))_2}$ and $\E{X_r(t)}^2$ agree, and
$$
\Var{X_r(t)} = \E{(X_r(t))_2} + \E{X_r(t)} - \E{X_r(t)}^2 = o(\E{X_r(t)}^2).
$$
We apply Chebyshev's inequality with some $\om$ tending to infinity sufficiently slowly:
$$
\Prob{|X_r(t) - \E{X_r(t)}| \geq \frac{\E{X_r(t)}}{\om^{1/2}}} \leq \frac{\Var{X_r(t)} \om}{\E{X_r(t)}^2} = o(1).
$$
Finally, if $n\d^{r/2} = o(1)$ then $\E{X_r(t)} \leq n\d^{r/2} = o(1)$ and Markov's inequality shows that $|X_r(t)| = 0$ w.h.p.
\end{proof}

Lemma \ref{lem:Xr} relates the number of unvisited edges to the number of unvisited vertices: we expect $|X_r(t)| = n-s$ to occur when $t \sim \left(1 - \frac{s}{n}\right)^{2/r}$. This heuristically explains why $C_E^b(G_r) \sim \frac{r}{2} C_V^b(G_r)$. Detailed calculations for the vertex cover time are carried out for $r = 3$ in \cite{CFJ}, and the calculations for larger $r$ are identical.

\section{The green edges}\label{sec:greenedges}

Let $\F(t)$ denote the number of green edges in $W(t)$.
\begin{lemma}\label{lem:greenedges}
Let $0 < \e < r-2$ and define
$$
\d_\e = \bfrac{\log^4 n}{n}^{\frac{r-1}{r+\e}}, \quad t_\e = (1-\d_\e)\frac{rn}{2}.
$$
Then with high probability, $\F(t) \geq n\d^{\frac{1+\e}{r-1}}$ for all $t_1\leq t \leq t_\e$.
\end{lemma}
With $\e > 0$ small enough so that $(r-1)/(r+\e) > 3/4$ and $(1+\e)/(r-1) < 1/2$, Lemma \ref{lem:greenedges} implies statement (ii) of Lemma \ref{lem:concentration}: if $t_1 \leq t \leq t_3$ then w.h.p., $\F(t) \geq n\d^{1/2}$.
\begin{proof}[Proof of Lemma \ref{lem:greenedges}]
Firstly, let us see how $\F(t)$ changes with time. Fix $\e_1 > 0$ such that
\begin{equation}\label{eq:eps1def}
\frac{1}{(1-\e_1)(r-1)} < \frac{1+\e}{r-1},
\end{equation}
let
$$
\XX(t) = \left\{X_1^g(t) \geq (1-\e_1)(rn-2t)\right\}
$$
and let $\ind{t}$ denote the indicator variable for $\XX(t)$. We note that with $\l = 1/\log n$, by Corollary \ref{cor:X1g}
\begin{equation}\label{eq:etadef}
\Prob{\ol{\XX(t)}} \leq \frac{\E{e^{-\l X_1^g(t)}}}{e^{-\l (1-\e_1)(rn-2t)}} \leq \exp\left\{-\frac{\e_1 n\d_\e}{\log n}\right\} =: \eta,
\end{equation}
for any $t \leq t_\e$.
\begin{claim}\label{cl:mgf}
For $0 < \th \leq \d_\e\log^{-2}n$, $\e_1 > 0$ and $t_0 \leq t \leq t_\e$,
$$
\E{e^{-\th(\F(t+1) - \F(t))}\ind{t}\mid [W(t)]} \leq \exp\left\{\frac{2\th\F(t)}{(1-\e_1)(r-1)(rn-2t)}(1+O(\g))\right\}\ind{t},
$$
with $\g = o(\log^{-1} n)$.
\end{claim}

\begin{proof}
Condition on a $[W(t)]$ such that $X_1^g(t) \geq (1-\e_1)(rn-2t)$. If the next edge is added without entering a blue walk, then $\F(t+1) = \F(t) + 1$. So,
$$
\Prob{\F(t + 1) - \F(t) = 1 \mid [W(t)]} = 1 - \frac{X_1(t)}{rn-2t}.
$$
Suppose the new edge chooses a vertex of $X_1(t)$, thus entering a blue walk. We may view this as a walk on $[W(t)]$, and any time a green edge is traversed, we ask if the green edge in $[W(t)]$ contains a green link in $W(t)$, in which case the blue walk ends. If not, the green edge turns blue and $\F$ decreases by one.

There are $L(t) = \frac{r-1}{2}X_1^g(t)$ green links, distributed into the $\F(t)$ green edges by a P\'olya urn process as discussed in Section \ref{sec:uniform}. Suppose $e_1,e_2,\dots,e_\ell$ are green edges in $[W(t)]$, and let $K_1,K_2,\dots,K_\ell$ be the lengths of the corresponding paths in $W(t)$, corresponding to the first $\ell$ entries of a vector $(k_1,\dots,k_\f)$ drawn uniformly at random from all vectors with $k_i\geq 1$ and $\sum_{i=1}^\f k_i = \F(t)$. The probability that none of the $\ell$ edges contains a green link is exactly
$$
\Prob{K_i = 1 \text{ for } i = 1,2,\dots,\ell} = \prod_{i=1}^\ell \frac{\binom{\F-i-1}{\f-i-1}}{\binom{\F-i}{\f-i}} 
= \prod_{i=1}^\ell\left(1 - \frac{L(t)}{\F(t)-i}\right)\leq \left(1 - \frac{L(t)}{\F(t)}\right)^\ell.
$$
This shows that the number of green edges visited before discovering a green link can be bounded by a geometric random variable. If a green edge is visited without a discovery, that edge turns blue. Note that the blue walk may also end when a vertex of $X_i^b$ is found for some $i\geq 1$; we are upper bounding the number of green edges visited.

So in distribution,
$$
\F(t+1) - \F(t) \stackrel{d}{=} 1 - B\bfrac{X_1(t)}{rn-2t} R_t
$$
where $B(p)$ denotes a Bernoulli random variable taking value $1$ with probability $p$, and $R_t$ is stochastically dominated above by a geometric random variable with success probability $L(t) / \F(t)$. The two random variables on the right-hand side are independent. So
$$
\E{e^{-\th(\F(t+1)-\F(t))}\mid [W(t)]} = e^{-\th}\left(1 - \frac{X_1(t)}{rn-2t} + \frac{X_1(t)}{rn-2t}\E{e^{\th R_t}\mid [W(t)]}\right)
$$
The map $x\mapsto e^{\th x}$ is increasing for $\th > 0$, so we can couple $R_t$ to a geometric random variable $S_t$ with success probability $L(t)/\F(t)$ in such a way that
$$
\E{e^{\th R_t}\mid [W(t)]} \leq \E{e^{\th S_t}\mid [W(t)]}.
$$
As $S_t$ is geometrically distributed and $X_1^g(t) \geq (rn-2t)/2$ by conditioning on $\XX(t)$,
$$
\E{e^{\th S_t}\ \middle|\ [W(t)]} = 1 + \th\frac{\F(t)}{L(t)} - O\bfrac{\th^2 \F(t)^2}{L(t)^2} = 1 + \th\frac{\F(t)}{L(t)} (1 + O(\g)).
$$
Conditioning on $X_1^g(t) \geq (1-\e_1)(rn-2t)$ implies that $L(t) = \frac{r-1}{2}X_1^g(t) = \OM(n\d)$, so
$$
\g := \th\frac{\F(t)}{L(t)} \leq \d_\e\log^{-2}n \frac{n}{\OM(n\d_\e)} = o(\log^{-1} n).
$$
We also have $X_1^b(t) \leq rn-2t - X_1^g(t) \leq \e_1(rn-2t)$, so
$$
\frac{X_1(t)}{L(t)} = \frac{X_1^g(t)}{L(t)} + \frac{X_1^b(t)}{L(t)} \leq \frac{2}{r-1} + \frac{\e_1(rn-2t)}{(1-\e_1)\frac{r-1}{2}(rn-2t)} = \frac{2}{(1-\e_1)(r-1)}.
$$
So for $[W(t)]\in \XX(t)$,
\begin{align}
& \E{e^{-\th(\F(t+1)-\F(t))}\ind{t}\mid [W(t)]} \\
& \leq e^{-\th}\left(1 - \frac{X_1(t)}{rn-2t} + \frac{X_1(t)}{rn-2t}\left(1 + \th\frac{\F(t)}{L(t)}(1 + O(\g))\right)\right) \\
& \leq \left(1 + \th\frac{2\F(t)}{(1-\e_1)(r-1)(rn-2t)}(1 + O(\g))\right) \\
& \leq \exp\left\{\frac{2\th\F(t)}{(1-\e_1)(r-1)(rn-2t)}(1+O(\g))\right\}.
\end{align}
\end{proof}
Define for $0 < \th = o(1)$,
$$
f_t(\th) = \E{e^{-\th \F(t)}\ind{t}}.
$$
As $\F(t) \geq L(t) = \frac{r-1}{2}X_1^g(t)$ we have for $0 < \th = o(1)$, by Corollary \ref{cor:X1g},
\begin{equation}\label{eq:ft0bound}
f_{t_0}(\th) \leq \E{e^{-\th \F(t_0)}} \leq \E{e^{-\th \frac{r-1}{2}X_1^g(t)}} = \exp\left\{-\th \frac{r-1}{2} rn\d_0(1 + o(1))\right\}.
\end{equation}
Claim \ref{cl:mgf} shows that for $t_0 \leq t < t_\e$,
\begin{align}
f_{t+1}(\th) & = \E{e^{-\th\F(t + 1)}\ind{t}} + \E{e^{-\th\F(t + 1)}(\ind{t+1}-\ind{t})} \\
& \leq \E{e^{-\th\F(t)}\E{e^{-\th(\F(t+1) - \F(t))}\ind{t}\mid [W(t)]}} + \E{\ind{t+1}} \\
& \leq \E{\exp\left\{-\th\F(t)\left(1 - \frac{2(1 + O(\g))}{(1-\e_1)(r-1)(rn-2t)}\right)\right\}} + \eta \\
& = f_t\left(\th\left(1 - \frac{2(1+O(\g))}{(1-\e_1)(r-1)(rn-2t)}\right)\right) + \eta
\end{align}
where $\eta = \exp\{-\e_1 n\d_\e/\log n\}$ is an upper bound for $\Prob{\ol{\XX(t+1)}}$, as defined in \eqref{eq:etadef}. As $\g = o(\log^{-1} n)$, repeating the calculations in \eqref{eq:exbound} and \eqref{eq:mbound}, we have
$$
\prod_{s = t_0}^{t-1} \left(1 - \frac{2(1 + O(\g))}{(1-\e_1)(r-1)(rn-2s)}\right) \sim \bfrac{rn-2t}{rn-2t_0}^{\frac{1}{(1-\e_1)(r-1)}}.
$$
It follows by induction and from \eqref{eq:ft0bound} that if $L(t) = n\d^{\frac{1+\e}{r-1}}$,
\begin{align} 
f_t(\th) & \leq f_{t_0}\left(\th\prod_{s=t_0}^{t-1} \left(1 - \frac{2(1+O(\g))}{(1-\e_1)(r-1)(rn-2s)}\right)\right) + (t-t_0)\eta \\
& \leq \exp\left\{-\th rn\d_0\bfrac{\d}{\d_0}^{\frac{1}{(1-\e_1)(r-1)}}\right\} + (t-t_0)\eta
\end{align}
Now, $\e_1$ was chosen in \eqref{eq:eps1def} to satisfy $1/(1-\e_1)(r-1) < (1+\e)/(r-1)$. The $\d_0 = 1/\log\log n$ factors are insignificant compared to those involving $\d \leq \log^{-1/2} n$, and we have
$$
n\d_0\bfrac{\d}{\d_0}^{\frac{1}{(1-\e_1)(r-1)}} > n\d^{\frac{1+\e}{r-1}} = L(t),
$$
which implies
$$
f_t(\th) \leq e^{-r \th L(t)} + n\eta
$$

Now, setting $\th = \d_\e\log^{-2} n$, using the bound $\ind{\{X > a\}} \leq X/a$,
\begin{align}
\Prob{\F(t) < L(t)} & \leq \Prob{\ol{\XX(t)}} + \Prob{\F(t) < L(t),\ \XX(t)}  \\
& \leq \eta + \E{\ind{\left\{e^{-\th\F(t)} > e^{-\th L(t)}\right\}}\ind{t}} \\
& \leq \eta + e^{\th L(t)} f_t(\th) \\
& = O(ne^{\th L(t)}\eta) + e^{-\th(r-1) L(t)}. \label{eq:part1}
\end{align}
We bound the two terms in \eqref{eq:part1} separately. Firstly,
\begin{equation}\label{eq:part2}
ne^{\th L(t)}\eta = n \exp\left\{\frac{\d_\e}{\log^2 n} n \d^{\frac{1+\e}{r-1}} - \frac{\e_1 n\d_\e}{\log n}\right\} \leq n \exp\left\{n\d_\e\left(\frac{\d_1^{\frac{1+\e}{r-1}}}{\log^2 n} - \frac{\e_1}{\log n}\right)\right\} 
\end{equation}
and as $\d = o(1)$ and $n\d_\e / \log n = \widetilde{\OM}\left(n^{\frac{1+\e}{r+\e}}\right)$, we have $ne^{\th L(t)}\eta = o(n^{-1})$. Secondly, for $\d \geq \d_\e = (n^{-1}\log^4 n)^{(r-1)/(r+\e)}$,
\begin{equation}\label{eq:part3}
e^{-\th(r-1) L(t)} = \exp\left\{-(r-1)\frac{\d_\e}{\log^2 n} n \d^{\frac{1+\e}{r-1}}\right\} \leq \exp\left\{- \frac{n \d_\e^{\frac{r+\e}{r-1}}}{\log^2 n}\right\} = e^{-\log^2 n},
\end{equation}
so combining \eqref{eq:part1}, \eqref{eq:part2} and \eqref{eq:part3}, we conclude
$$
\Prob{\exists t_1 \leq t \leq t_\e : \F(t) < L(t)} \leq n \left(o(n^{-1}) + O(e^{-\log^2 n})\right) = o(1).
$$
\end{proof}

\end{document}